\documentclass[12pt,a4paper]{amsart}

\usepackage[utf8]{inputenc}
\usepackage{amsmath,amsfonts,amssymb,amsthm,mathtools}
\usepackage{hyperref}
\usepackage{url}
\usepackage{xcolor}
\usepackage{verbatim}
\usepackage{soul}

\newtheorem{theorem}{Theorem}[section]
\newtheorem{corollary}[theorem]{Corollary}
\newtheorem{lemma}[theorem]{Lemma}
\newtheorem{remark}[theorem]{Remark}

\makeatletter
\@namedef{subjclassname@2020}{%
  \textup{2020} Mathematics Subject Classification}
\makeatother

\newcommand{\Z}{\mathbb{Z}}
\newcommand{\Q}{\mathbb{Q}}
\newcommand{\R}{\mathbb{R}}
\newcommand{\Fq}[1][]{\mathbb{F}_{q^{#1}}}
\newcommand{\Roots}{\mathrm{Roots}}

\begin{document}


\title[Abelian varieties with many rational points]{Weil polynomials of abelian varieties over finite fields with many rational points}
\date{}
\author[E. Berardini]{Elena Berardini}
\address{Laboratoire d'informatique de l'École polytechnique (LIX), CNRS, École polytechnique, Institut Polytechnique de Paris, 91120 Palaiseau, France}
\email{elena\_berardini@hotmail.it}
\author[A. J. Giangreco-Maidana]{Alejandro~J.~Giangreco-Maidana}
\address{Universit{\'e} Polytechnique Hauts-de-France, Laboratoire de Math{\'e}matiques pour l'Ing{\'e}nieur (LMI), FR CNRS 2956, F-59313 Valenciennes, France and 
\indent Universidad Nacional de Asunci{\'o}n, Facultad de Ingenier{\'i}a, Paraguay}
\email{agiangreco@ing.una.py}


\maketitle

\begin{abstract}
We consider the finite set of isogeny classes of $g$--dimensional abelian varieties defined over the finite field $\Fq$ with endomorphism algebra being a field. We prove that the class within this set whose varieties have maximal number of rational points is unique, for any prime even power $q$ big enough and verifying mild conditions. We describe its Weil polynomial and we prove that the class is ordinary and cyclic outside the primes dividing an integer that only depends on $g$. In dimension $3$, we prove that the class is ordinary and cyclic and give explicitly its Weil polynomial, for any prime even power $q$.

\end{abstract}

\section{Introduction}
In the present paper we study abelian varieties defined over finite fields with groups of rational points of large cardinality, and their cyclicity. Arithmetic properties of abelian varieties and their groups of rational points are of large interest and were studied since the last century (\cite{Waterhouse,Ruck1990,DIPIPPO1998426,HOWE2002139,Aubry}). A renewed interest on the structure and the cardinality of the groups of rational points on abelian varieties is due to their applications to information theory. For instance, cyclic subgroups of abelian varieties are used in cryptography (\cite{Koblitz2000,RS02}), while abelian surfaces with group of rational points of large cardinality allow to construct long algebraic geometry codes (\cite{MYFE}).

A powerful tool to study isogeny classes of abelian varieties comes from the elegant Honda--Tate theory (\cite{Tate1971}) via Weil polynomials. These last carry a lot of arithmetic and geometric information on the varieties, \textit{e.g.}~their number of rational points.  For small dimensions, there exist precise criteria to decide whether a polynomial is the Weil polynomial of an isogeny class. On the other hand, not much is known in general dimension and this is one of the main obstacles in the study of isogeny classes.

The cardinality of the group of rational points of an abelian variety is an invariant of its isogeny class. The isogeny class of $g$--dimensional abelian varieties defined over $\Fq$ with the maximal number of rational points contains the $g$--power of the maximal elliptic curve defined over the same finite field. Therefore, its Weil polynomial is easily given. It is then natural to study the maximal isogeny class of abelian varieties under some restrictions. 

The endomorphism algebra of an abelian variety is an invariant under isogenies. In this paper, we focus on isogeny classes whose endomorphism algebra is a field. This condition is stronger than being simple, and it is actually equivalent to be of type I or IV with $d=1$ in Albert's classification (\cite{A34,A35}). First, note that among simple abelian varieties these varieties are the most common ones. Secondly, this allows us to work with irreducible Weil polynomials. 

Let us consider the finite set  of isogeny classes of $g$-dimensional abelian varieties defined over $\Fq$ that have endomorphism algebra which is a field. Besides the case of dimension one, two not isogenous abelian varieties defined over the same field can have the same number of rational points. If there is only one isogeny class of abelian varieties within the previous finite set having the maximal number of rational points, we define it as the \textbf{maximal class}
and we denote it by $\mathcal{I}_{\text{max}}^0(g,q)$. From now on whenever the notation $\mathcal{I}_{\text{max}}^0(g,q)$ is used, the uniqueness of the class is understood.

The aim of the present paper is to characterise the class $\mathcal{I}_{\text{max}}^0(g,q)$. For any positive integer $g$, we prove that $\mathcal{I}_{\text{max}}^0(g,q)$ exists for any prime even power $q$ big enough  and verifying some mild condition, and we describe its Weil polynomial (Theorem \ref{th:main}). As a consequence, we deduce that the class is ordinary and cyclic outside primes dividing an integer which does not depend on $q$ (Corollary \ref{cor:main}). Here \emph{cyclic} refers to the structure of the groups of rational points of the varieties inside the class.

Our approach is linked to the theory of totally positive algebraic integers with minimal trace. As it is for abelian varieties, such algebraic integers have been extensively studied during the last century (\cite{Siegel1945,Smyth1984,Boyd1985}) and are still an active domain of research (\cite{ElOtmani2014,MckeeSmyth2004_ANTS}). However, much is still unknown about them. For instance, it is not even known the optimal lower bound for the absolute trace (see \cite{Siegel1945,Smyth1984}). The connection within algebraic integers of small trace and the groups of rational points on algebraic varieties was previously pointed out by Serre (see for instance \cite{Aguirre2008} and the introduction of \cite{Smyth1984_AIF}).

For a positive integer $g$, we define $\mathfrak{f}_g$ to be the minimal polynomial of a totally positive algebraic integer of degree $g$ with minimal trace, and which is maximal in a sense that we make precise in Lemma \ref{lemma:key} (see also Table \ref{table:f_g} for some examples). 
The main result of the paper is the existence of a polynomial that ``parametrizes'' the family of isogeny classes we are interested in, as stated in the following theorem.

\begin{theorem}\label{th:main}
Let $g$ be a positive integer and let $r_1,\dots,r_g$ be the roots of $\mathfrak{f}_g$.
There exists a real number $c_g$ such that for every even power $q>c_g$ of a prime, coprime with $\mathfrak{f}_g(0)$, the maximal class $\mathcal{I}_{\text{max}}^0(g,q)$ exists, it is ordinary and has $h_g(t,\sqrt{q})$ as Weil polynomial, where
\begin{align*}
    h_g(t,X)\coloneqq\prod_{i=1}^g (t^2+(2X-r_i)t+X^2).
\end{align*}
In particular, the totally real subfield of the endomorphism algebra of $\mathcal{I}_{\text{max}}^0(g,q)$ is isomorphic to the extension of $\Q$ generated by one $r_i$. The isogeny class $\mathcal{I}_{\text{max}}^0(g,q)$ is also the maximal one among the simple ordinary isogeny classes of $g$--dimensional abelian varieties defined over $\Fq$.
\end{theorem}

The group structure of abelian varieties is not an invariant of the isogeny class (see \cite[Th.~3]{ruck1987note}). However, we can define the cyclicity for isogeny classes. Let $A(k)_{\ell}$ be the $\ell$--primary component of the group $A(k)$ of $k$--rational points on the abelian variety $A$. An isogeny class $\mathcal{A}$ defined over a field $k$ is said to be $\ell$\textbf{--cyclic} if $A(k)_{\ell}$ is cyclic for any $A\in\mathcal{A}$. The second author characterized in \cite{GiangrecoMaidana} the cyclicity of an isogeny class $\mathcal{A}$. This characterisation is based on its Weil polynomial $f_\mathcal{A}$ (see the \emph{cyclicity criterion} (\ref{thm:l-cyclic})). 
As an interesting consequence of Theorem \ref{th:main}, we obtain the following result.
\begin{corollary}\label{cor:main}
With the same notation and under the same hypotheses of Theorem \ref{th:main}, we have that the isogeny class $\mathcal{I}_{\text{max}}^0(g,q)$ is $\ell$--cyclic for primes $\ell$ not dividing 
\[
N_g\coloneqq\mathfrak{f}_g(4) \mathfrak{f}_g(0) \Delta_g,
\]
where $\Delta_g$ is the discriminant of $\mathfrak{f}_g$.
\end{corollary}

For small dimensions, we are able to give $h_g(t,X)$, $c_g$ and $N_g$ explicitly. As an example, from the well--known theory on elliptic curves and their Weil polynomials, it follows that 
\[
h_1(t,X)=t^2+(2X-1)t+X^2,
\]
and Theorem \ref{th:main} holds for all even powers of a prime, so we can take $c_1=0$. Moreover, we have that $\mathfrak{f}_1(t)=t-1$ and $N_1=-3$. We easily verify that $\mathcal{I}_{\text{max}}^0(1,7^2)$ is not $3$--cyclic and, as a matter of fact, there are infinitely many $q$ for which $\mathcal{I}_{\text{max}}^0(1,q)$ is not $3$--cyclic. For $g\leq 3$ we have $N_2=5^2$ (see the beginning of Section \ref{section:threefolds}) and $N_3=-7^3$ (see Remark \ref{Ng}). Moreover, using the explicit description of Weil polynomials of abelian surfaces given in \cite{Ruck1990}, the second author found (\cite{GiangrecoMaidana}) that the class $\mathcal{I}_{\text{max}}^0(2,q)$ is everywhere cyclic, for any prime even power $q$. In the present paper, we use the explicit description of Weil polynomials of degree $6$ given in \cite{Haloui} to prove that $\mathcal{I}_{\text{max}}^0(3,q)$ is everywhere cyclic (see Theorem \ref{ourfirsttheorem}), for $q$ an even power of a prime.

Regarding odd powers of primes, some natural questions arise. Is there a polynomial that parametrizes $\mathcal{I}_{\text{max}}^0(g,q)$? Can we have arbitrarily large primes $\ell$ such that $\mathcal{I}_{\text{max}}^0(g,q)$ is not $\ell$--cyclic for some odd power $q$?
The answer to the first question is no, at least for dimension one. The existence of such a polynomial would imply that $\lfloor 2\sqrt{q} \rfloor$ can be written as a polynomial, which is not the case. The second question seems to be much more complicated. Even when $g=1$ we are not able to give an answer. 

The rest of the paper is organized as follows. In Section \ref{section:ab_var} we recall some theory about abelian varieties defined over finite fields. In Section \ref{section:main} we prove our main results and Section \ref{section:threefolds} is devoted to the case of abelian varieties of dimension $3$.

\section{Abelian varieties over finite fields}\label{section:ab_var}
This section contains well--known facts about abelian varieties defined over finite fields. We state here the definitions and the results we shall use later, and refer the reader to \cite{Milne_AV,Waterhouse} for further details.

Let $q=p^r$ be a power of a prime. A $q$\textbf{--Weil number} is an algebraic integer with all its conjugates of absolute value $\sqrt{q}$. A monic polynomial with integer coefficients and with all its roots being $q$--Weil numbers is called a $q$\textbf{--Weil polynomial}. Over real numbers, $q$--Weil polynomials have the form
\begin{align}\label{eq:Weil_p_real}
\prod_{i=1}^g (t^2+x_i t+q), \quad g\in\mathbb{N}, x_i\in \R \text{ and } |x_i|\leq 2\sqrt{q},
\end{align}
provided that the roots are complex numbers. If a $q$--Weil polynomial has real roots, then it has the form $(t\pm\sqrt{q})^2$ or $(t^2-q)^2$ (see \cite[p.~528]{Waterhouse}). 
We will need the following elementary lemma.
\begin{lemma}\label{lemma:irred}
Let $\sqrt{q}$ be an integer. For $i\in\{1,\dots,g\}$, let $x_i:=2\sqrt{q}-s_i$, for some $s_i\in\R$. The following statements are equivalent:
\begin{enumerate}
    \item\label{pol1} the $q$--Weil polynomial (\ref{eq:Weil_p_real}) is irreducible over $\Q$;
    \item\label{pol2} the polynomial $\prod_{i=1}^{g} (t-s_i)$ is irreducible over $\Q$;
    \item\label{pol3} the polynomial $\prod_{i=1}^g (t^2+(2X-s_i)t+X^2)\in\Z[t,X]$ is irreducible over $\Q$.
\end{enumerate}
\end{lemma}
\begin{proof}
Consider the polynomial \eqref{pol1}. For any subset $I\subseteq \{1,\dots,g\}$, the polynomial
\[
\prod_{i\in I} (t^2+x_i t+q)
\]
has the degree one elementary symmetric polynomial  $\sum_{i\in I} s_i$ as one of its coefficients, up to a sum of an integer. Any other elementary symmetric polynomial in the $s_i$ appears as some of its coefficients, up to a sum of an integer and up to a sum of the elementary symmetric polynomials of less degree multiplied by an integer. The same occurs when considering products of factors of the polynomials \eqref{pol2} or \eqref{pol3} of the statement.
Hence, the reducibility of one of the polynomials in the statement implies that all the elementary symmetric polynomials in the $\{s_i\}_{i\in I}$ are integers for some $I$. This entails the reducibility of the other two polynomials.
\end{proof}

Let $A$ be an abelian variety of dimension $g$ defined over $\Fq$, a finite field with $q$ elements. The Frobenius endomorphism of $A$ acts on its Tate module by a semisimple linear operator, and its characteristic polynomial $f_{A}(t)$ is called the \textbf{Weil polynomial of $A$}.  Weil proved that $f_{A}(t)$ is a $q$--Weil polynomial.  In \cite{tate1966}, Tate proved that two abelian varieties $A$ and $B$ are isogenous if and only if $f_A(t)=f_B(t)$. Thus, it makes sense to consider the Weil polynomial $f_\mathcal{A}(t)$ of an isogeny class $\mathcal{A}$ as being the Weil polynomial of some (and thus any) abelian variety in $\mathcal{A}$. It has the form 
\begin{align}
f_\mathcal{A}(t)&=t^{2g}{+}a_1 t^{2g-1}{+}\dots {+} a_g t^g {+} a_{g-1} q t^{g-1}{+}\dots a_1 q^{g-1} t {+} q^{g} \\
&= \prod_{i=1}^g (t^2+x_i t+q). \label{eq:Weil_p_real_A}
\end{align}

An abelian variety $A$ is said to be \textbf{simple} if it has no proper and nonzero sub--abelian varieties. Being simple is a property of the isogeny class. The Weil polynomial of an abelian variety is the product of the Weil polynomials of its simple factors. The Weil polynomial of a simple isogeny class is equal to $f_A(t)=h(t)^e$ for an irreducible polynomial $h(t)\in\Z[t]$ and a positive integer $e$. Thus we have a map $\Phi$ that associates to every isogeny class its Weil polynomial. The Honda--Tate theory says that this map gives a bijection between isogeny classes of simple abelian varieties defined over $\Fq$ and conjugacy classes of $q$--Weil numbers. 
A $q$--Weil polynomial of degree $2g$ is said to be \textbf{ordinary} if the coefficient of degree $g$ of $f_A(t)$, \textit{i.e.}~the \emph{middle coefficient}, is not divisible by $p$.
The variety $A$ is said to be \textbf{ordinary} if its $p$--rank is $g$. This is equivalent to $A$ having ordinary Weil polynomial. It follows that being ordinary is also a property of the isogeny class. Not every irreducible $q$--Weil polynomial is the Weil polynomial of a simple isogeny class of abelian varieties. However, when restricted to ordinary isogeny classes, the previous map $\Phi$ is a bijection between isogeny classes of simple ordinary abelian varieties defined over $\Fq$ and irreducible $q$--Weil polynomials (this is the ordinary Honda--Tate theory, see \cite[Th.~3.3]{Howe1995}).

The cardinality of the group $A(k)$ of $k$--rational points on $A$ equals $f_A(1)$, hence it is an invariant of the isogeny class. Therefore, by abuse of language, we can talk about the number of rational points of an isogeny class, meaning the number of rational points on any variety inside the class.

Following the definition of cyclicity for isogeny classes, Theorem 2.2 of \cite{GiangrecoMaidana} gives the \emph{cyclicity criterion}: for any prime $\ell$ and for any isogeny class $\mathcal{A}$ we have that
\begin{align}\label{thm:l-cyclic}
    \mathcal{A} \text{ is } \ell\text{--cyclic if and only if } \ell\text{ does not divide } (\widehat{f_\mathcal{A}(1)},f_\mathcal{A}'(1)),
\end{align}
where for an integer $z$, $\widehat{z}$ denotes the quotient of $z$ by its radical, and $(z_1,z_2)$ is the greatest common divisor of $z_1$ and $z_2$. Corollary \ref{cor:main} uses a weaker version of the cyclicity criterion, namely $\ell\nmid (f_\mathcal{A}(1),f_\mathcal{A}'(1))$ implies that the isogeny class is $\ell$--cyclic.

Abelian varieties belonging to the same isogeny class $\mathcal{A}$ have isomorphic endomorphism algebras.
The latter are commutative if and only if $f_\mathcal{A}(t)$ has no multiple roots (see \cite[Th.~2]{Tate1971}). A simple abelian variety with commutative endomorphism algebra has irreducible Weil polynomial (see \cite[Ch.~2]{Waterhouse}). In particular, its endomorphism algebra is a CM--field. Its maximal real subfield is generated by one of the $x_i$ in the factorisation of Equation (\ref{eq:Weil_p_real_A}).
\section{Maximal isogeny classes and cyclicity}\label{section:main}
This section is devoted to the proof of the main results of the paper, as stated in the introduction.

We define $\mathcal{F}_g$ to be the subset of $\mathbb{Z}[t]$ consisting of monic polynomials of degree $g$ that are irreducible over $\mathbb{Q}$ and with all positive real roots. The key ingredient in the proof of Theorem \ref{th:main} is the existence of an irreducible polynomial as given in the following lemma.
\begin{lemma}\label{lemma:key}
Let $g$ be a positive integer. Then there exists a polynomial $\mathfrak{f}_g \in \mathcal{F}_g$ and a real number $n_g$ such that $\mathfrak{f}_g(t)>f(t)$ for any other $f\in\mathcal{F}_g$ and for any $t>n_g$.
\end{lemma}
\begin{proof}
The nonemptiness of $\mathcal{F}_g$ is equivalent to the existence of totally real algebraic integers of degree $g$, which follows from  \cite[p.~48 and \S 3]{lenstra1974simple}. We define a total order relation in $\mathcal{F}_g$ as follows: $f_1\leq f_2$ if and only if either $f_2-f_1$ has positive leading coefficient or $f_2$ and $f_1$ are the same polynomial. Let $\mathcal{F}_g^{min}$ be the subset of $\mathcal{F}_g$ of polynomials with minimal trace. Since the coefficients of all the polynomials in $\mathcal{F}_g^{min}$ are bounded, this set is finite. Consequently, we can take the maximal element $\mathfrak{f}_g$ of $\mathcal{F}_g^{min}$, which is clearly the maximal element of $\mathcal{F}_g$ as well. 
The subset of polynomials of $\mathcal{F}_g$ with roots $\leq \max \{\Roots(\mathfrak{f}_g)\}$ is also finite. This justify the existence of the $n_g$.
\end{proof}

We are now ready to prove our main result.
\begin{proof}[Proof of Theorem \ref{th:main}]
Let $n_g$ be given by Lemma \ref{lemma:key}. We write  $\mathfrak{f}_g(t)=\prod^g_{i=1}(t-r_i)$ where we can suppose $r_1<\dots < r_g$. We recall that
\begin{align}
h_g(t,X)\coloneqq\prod_{i=1}^g(t^2+(2X-r_i)t+X^2).
\end{align}

We set $c_g$ such that $r_g\leq 4\sqrt{c_g}$  and $n_g \leq (\sqrt{c_g}+1)^2 $. For every even power $q>c_g$, we claim that $h_g(t,\sqrt{q})$ is an irreducible ordinary $q$--Weil polynomial. First, note that the first condition on $c_g$ ensures that we have $r_i<4\sqrt{q}$ for every $i=1,\dots,g$, thus $h_g(t,\sqrt{q})$ is a $q$--Weil polynomial. 
Secondly, $h_g(t,\sqrt{q})$ is ordinary since its middle coefficient is congruent modulo $\sqrt{q}$ to \[\prod_{i=1}^g (-r_i)= \mathfrak{f}_g(0),\]
and by hypothesis $\mathfrak{f}_g(0)\not\equiv 0$ modulo $\sqrt{q}$.
Finally, it is irreducible because if not, it would contradict the irreducibility of $\mathfrak{f}_g$ (see Lemma \ref{lemma:irred}). Thus, $h_g(t,\sqrt{q})$ corresponds to an ordinary simple isogeny class by the Honda--Tate theory for ordinary varieties (see \cite[Th.~3.3]{Howe1995}). Note that any other isogeny class with endomorphism algebra which is a field and corresponding to the Weil polynomial $\prod_{i=1}^g(t^2+(2\sqrt{q}-s_i)t+q)$ for some real numbers $s_i$, will produce a polynomial $\prod_{i=1}^g(t-s_i)$ in $\mathcal{F}_g$. The cardinality of the group of rational points of the class with Weil polynomial $h_g(t,\sqrt{q})$ equals $h_g(1,\sqrt{q})=\mathfrak{f}_g((\sqrt{q}+1)^2)$.  The maximality is then a consequence of Lemma \ref{lemma:key}, since we have chosen $c_g$ such that $n_g \leq (\sqrt{c_g}+1)^2 $. The uniqueness of the class follows by construction. This class is then $\mathcal{I}_{\text{max}}^0(g,q)$. The statement about the totally real subfield of the endomorphism algebra of $\mathcal{I}_{\text{max}}^0(g,q)$ is clear from the construction. The last part of the statement follows from the ordinary Honda--Tate theory.
\end{proof}

Now we can prove Corollary \ref{cor:main}.
\begin{proof}[Proof of Corollary \ref{cor:main}]
Let $\ell$ be a prime dividing $h_g(1,\sqrt{q})$ and let $\mathfrak{P}$ be a prime of $\Q(r_1,\dots,r_g)$ lying over $\ell$. Hence $\mathfrak{P}$ divides $(q+2\sqrt{q}+1-r_{i_0})$ for some $i_0\in\{1,\dots,g\}$. Then
\[
\frac{\partial h_g}{\partial t} (1,\sqrt{q}) \equiv (2+2\sqrt{q}-r_{i_0}) \prod_{j\neq i_0} (q+2\sqrt{q}+1-r_j) \pmod{\mathfrak{P}}.
\]
In order for $\mathcal{I}_{\text{max}}^0(g,q)$ to not be $\ell$--cyclic, $\ell$ must divide $(\partial h_g/\partial t) (1,\sqrt{q})$ (see the \emph{cyclicity criterion} (\ref{thm:l-cyclic})). If this is the case, we  have the following two possible situations:

\begin{enumerate}
\item $\mathfrak{P}$ divides $(2+2\sqrt{q}-r_{i_0})$. Here, either $\mathfrak{P}|\sqrt{q}+1$, either $\mathfrak{P}|\sqrt{q}-1$. In the first case $\mathfrak{P}|r_{i_0}$, and in the second case $\mathfrak{P}|(4-r_{i_0})$;

\item $\mathfrak{P}$ divides $(q+2\sqrt{q}+1-r_{j_0})$ for some $j_0$. Here, we verify that $\mathfrak{P}|(r_{i_0}-r_{j_0})$.
\end{enumerate}

The first case corresponds to the term $\mathfrak{f}_g(0)\mathfrak{f}_g(4)$ of $N_g$, while the second one corresponds to $\Delta_g$. This concludes the proof.
\end{proof}

\begin{remark}
It is worth noting that Corollary \ref{cor:main} does not imply that $\mathcal{I}_{\text{max}}^0(g,q)$ is not $\ell$--cyclic for the primes $\ell$ dividing $N_g$. Indeed, it was proved for $g=2$ and we will prove it for $g=3$ (see Theorem \ref{ourfirsttheorem}) that $(\widehat{f_\mathcal{A}(1)},f_\mathcal{A}'(1))=1$, whereas $N_g>1$ (see the beginning of Section \ref{section:threefolds} and Remark \ref{Ng}).
\end{remark}

As pointed out in the Introduction, polynomials with minimal trace have been extensively studied in the literature. 
The set $\mathcal{F}_g^{min}$ is completely known for some values of $g$, thus we can easily deduce the corresponding $\mathfrak{f}_g$. Table \ref{table:f_g} summarizes the polynomials $\mathfrak{f}_g$ for $g\leq 7$ given in \cite{Smyth1984_AIF}. For $g=10$, from \cite{MckeeSmyth2004_ANTS} we get
$$
\begin{aligned}
\mathfrak{f}_{10}(t)=&t^{10}-18t^9+135t^8-549t^7+1320t^6-1920t^5+\\
&1662t^4-813t^3+206t^2-24t+1.
\end{aligned}
$$
Here we have that $\mathfrak{f}_{10}(4)=-191$ and $\Delta_{10}=983479472873=1567\times 627619319$.
Other authors give subsets of $\mathcal{F}_g^{min}$, so we cannot deduce $\mathfrak{f}_g$ from them.

\begin{table}[h!]
    \centering
\begin{tabular}{|c|c|c|c|}
\hline
$\mathfrak{f}_g(t)$ & $\mathfrak{f}_g(4)$ & $\mathfrak{f}_g(0)$ & $\Delta_g$ \\ \hline \hline 
$t-1$ & $3$ & -1 & $1$ \\ \hline
$t^2-3t+1$ & $5$ & 1 & $5$ \\ \hline
$t^3-5t^2+6t-1$ & $7$ & -1 & $7^2$ \\ \hline
$t^4-7t^3+14t^2-8t+1$ & $1$ & 1& $3^2\times 5^3$ \\ \hline
$t^5-9t^4+28t^3-35t^2+15t-1$ & $11$ & -1 & $11^4$ \\ \hline 
$t^6-11t^5+45t^4-84t^3+70t^2-21t+1$ & $13$ & 1 & $13^5$ \\ \hline
$t^7-13t^6+65t^5-157t^4+188t^3-102t^2+20t-1$ & $-17$ & -1 & $563\times 103651$ \\ 
\hline
\end{tabular}
\caption{Polynomials $\mathfrak{f}_g$ (\cite{Smyth1984_AIF}). }
    \label{table:f_g}
\end{table}

We are not able to find a $g$ and a $q$ for which there is no a unique maximal isogeny class within those with endomorphism algebra being a field. For the sets $\mathcal{F}_g$ that are known, $\mathfrak{f}_g(t)-f(t)$ is strictly positive for all $f\in\mathcal{F}_g$ and $t> 9$, which is what we need for the existence of $\mathcal{I}_{\text{max}}^0(g,q)$ for $q\geq 4$. Note that all the known $\mathfrak{f}_g$ have constant term equal to $\pm 1$. This is not true for all $f\in\mathcal{F}_g^{min}$ (see \cite[\S 3]{Aguirre2008} for a counterexample), but it still could be the case for $\mathfrak{f}_g$.

\subsection{Minimal isogeny classes}
Our results are completely analogous if we want to study the minimal (again, in the sense of the number of rational points) isogeny class within those with endomorphism algebra being a field. More precisely, the polynomial defining this class is
\begin{align}\label{hminus}
h^{-}_g(t,X)=\prod_{i=1}^g (t^2+(-2X-r_i)t+X^2),
\end{align}
where the $r_i$ are the roots of a polynomial $\mathfrak{f}_g^{-}$ chosen as follows. Consider the set $\mathcal{F}_g^{-}$ of irreducible monic polynomials of degree $g$ with integer coefficients and with all negative real roots. There exists a unique polynomial $\mathfrak{f}_g^{-}\in \mathcal{F}_g^{-}$ which is minimal: for any other $f\in\mathcal{F}_g^{-}$, the polynomial $f-\mathfrak{f}_g^{-}$ has a positive leading coefficient. Note that to every $f(t)=\prod_i(t-s_i)\in\mathcal{F}_g^{-}$ we can associate an isogeny class over $\Fq$ with Weil polynomial
\[
\prod_{i=1}^g (t^2+(-2\sqrt{q}-s_i)t+q),
\]
provided that the absolute values of the roots of $f(t)$ are $<4\sqrt{q}$. The number of rational points of the class equals $f((\sqrt{q}-1)^2)$.

In this context, we can give an example where the uniqueness fails when $q$ is not big enough. 
We check that $\mathfrak{f}_3^{-}(t)=t^3+5t^2+6t+1$ and $f(t)=t^3+6t^2+5t+1\in \mathcal{F}_3^{-}$ (see \cite{Smyth1984_AIF}). Taking $q=4$, we have $\mathfrak{f}_3^{-}(1)=f(1)=13$. This corresponds to the two minimal isogeny classes of threefolds \cite[\href{https://www.lmfdb.org/Variety/Abelian/Fq/3/4/ah_ba_acl}{Abelian variety isogeny class 3.4.ah\_ba\_acl}]{lmfdb_site} and \cite[\href{https://www.lmfdb.org/Variety/Abelian/Fq/3/4/ag_r_abj}{Abelian variety isogeny class 3.4.ag\_r\_abj}]{lmfdb_site} within those with endomorphism algebra being a field, from the LMFDB database.

\section{The case of threefolds}\label{section:threefolds}
In \cite[Th.~4.3]{GiangrecoMaidana} the second author proved that if $q$ is an even power of a prime, then the isogeny class $\mathcal{I}_{\text{max}}^0(2,q)$ is ordinary and cyclic, and has Weil polynomial
$t^4+at^3+bt^2+aqt+q^2$, with $a=4\sqrt{q}-3$ and $b=6q-6\sqrt{q}+1$. 
Using the notation introduced so far, this polynomial corresponds to the polynomial 
$
\mathfrak{f}_2(t)=t^2-3t+1
$ (see Table \ref{table:f_g}), and we have $N_2=5^2$. 
The only isogeny class of abelian surfaces with more rational points than $\mathcal{I}_{max}^0(2,q)$ and with commutative endomorphism algebra is $\mathcal{E}_{max-1}(q)\times \mathcal{E}_{max-2}(q)$, where $\mathcal{E}_{max-1}(q)$ and $\mathcal{E}_{max-2}(q)$ denote the isogeny classes of elliptic curves with Frobenius traces $2\sqrt{q}-1$ and $2\sqrt{q}-2$, respectively. Note that indeed $\mathcal{E}_{max-1}(q)$ equals $\mathcal{I}_{\text{max}}^0(1,q)$, but we decided to keep the notation which makes reference to the trace.
 
In this section we focus on isogeny classes of abelian varieties of dimension $3$ and obtain a similar characterisation for $\mathcal{I}_{\text{max}}^0(3,q)$.
The Weil polynomial of abelian threefolds has the following shape
\begin{equation}\label{weilpoly}
f(t)= t^6 + at^5 + bt^4 + ct^3 + qbt^2 + q^2at + q^3.
\end{equation}

Explicit description of which polynomials of the form (\ref{weilpoly}) appear as Weil polynomials of abelian threefolds can be found in \cite{Haloui}. We use this description to  completely characterize the class $\mathcal{I}_{\text{max}}^0(3,q)$ over any finite field $\Fq$ with $q$ an even power of a prime, as follows. 

\begin{theorem}\label{ourfirsttheorem}
Let $q$ be an even power of a prime. Then the maximal class $\mathcal{I}_{\text{max}}^0(3,q)$ exists, is ordinary, cyclic and has the following coefficients for its Weil polynomial:
\begin{align*}
a&=6\sqrt{q}-5;\\
b&= 15q - 20\sqrt{q}+6;\\
c&=20q\sqrt{q}-30q+12\sqrt{q}-1.
\end{align*}
Furthermore, Table \ref{more_rational_points} summarizes the only two isogeny classes of threefolds with more rational points than $\mathcal{I}_{\text{max}}^0(3,q)$, such that their endomorphism algebra is commutative. Those are $\mathcal{E}_{max-1}(q)\times \mathcal{I}_{\text{max}}^0(2,q)$ and $\mathcal{E}_{max-2}(q)\times \mathcal{I}_{\text{max}}^0(2,q)$, respectively. The first class is cyclic outside the $3$--primary component. The second one is cyclic.

\end{theorem}

\begin{table}[h!]
    \centering
\begin{tabular}{|c|c|c|c|}
    \hline
    $\textbf{a}$ & $\textbf{b}$ & $\textbf{c}$ &  \textbf{Class}  \\ \hline\hline
    $6\sqrt{q} - 1$ & - &  - & -  \\ \hline
    $6\sqrt{q} - 2$ & $15q - 8\sqrt{q} + 1$ & -  & -  \\ \hline
    $6\sqrt{q} - 3$ & $15q - 12\sqrt{q} + 2$ & -  & -  \\ \hline
     & $15q - 12\sqrt{q} + 1$ & -  & -  \\ \hline
    $6\sqrt{q} - 4$ & $15q - 16\sqrt{q} + 5$ & -  & -  \\ \hline
     & $15q - 16\sqrt{q} + 4$ & $20q\sqrt{q}-24q+8\sqrt{q}-1$ &  $\mathcal{E}_{max-1}(q)\times \mathcal{I}_{\text{max}}^0(2,q)$  \\ \hline
     & $15q - 16\sqrt{q} + 3$ & -  & -  \\ \hline
     & $15q - 16\sqrt{q} + 2$ & -  & -  \\ \hline
     & $15q - 16\sqrt{q} + 1$ & -  & -  \\ \hline
    $6\sqrt{q} - 5$ & $15q - 20\sqrt{q} + 8$ & -  & -  \\ \hline
     & $15q - 20\sqrt{q} + 7$ & $20q\sqrt{q}-30q + 14\sqrt{q} -2$  & $\mathcal{E}_{max-2}(q)\times \mathcal{I}_{\text{max}}^0(2,q)$  \\ \hline
     & $15q - 20\sqrt{q} + 6$ & $20q\sqrt{q} - 30q + 12\sqrt{q} - 1$  & $\mathcal{I}_{\text{max}}^0(3,q)$  \\ \hline
\end{tabular}
\caption{Isogeny classes of threefolds with more rational points than $\mathcal{I}_{\text{max}}^0(3,q)$, with commutative endomorphism algebra. The symbol - indicates when there is no Weil polynomial corresponding to an isogeny class whose endomorphism algebra is commutative.}
    \label{more_rational_points}
\end{table}

\begin{proof}
The proof of the first part can mainly be divided in two steps. First, we identify the unique Weil polynomial corresponding to the maximal class, for a fixed $q$. It is the one with maximal coefficients and that is irreducible. This is done by exhaustive research (see Table \ref{more_rational_points}) and using the bounds provided in \cite[Th.~1.1]{Haloui}. We consider the strict inequalities since equality is equivalent to have multiple roots (see \cite[Lem.~2.1.3]{DIPIPPO1998426}). This gives the polynomial in the statement of the theorem. It turns out that its coefficients can be written as polynomials in $\sqrt{q}$ (even if the bounds in \cite[Th.~1.1]{Haloui} are not polynomials). Considering these coefficients as polynomials in $X=\sqrt{q}$, we obtain an irreducible bivariate polynomial, that we denote by $h(t,X)$. By Lemma \ref{lemma:irred}, $h(t,\sqrt{q})$ is irreducible for every $\sqrt{q}$. Thus, it gives the desired polynomial for every even power $q$ of a prime, hence the existence of $\mathcal{I}_{\text{max}}^0(3,q)$. From the form of the middle coefficient $c$, we deduce that $\mathcal{I}_{\text{max}}^0(3,q)$ is ordinary.

Secondly, we consider $f(1)=h(1,X)$ and $f'(1)=\frac{\partial h}{\partial t}(1,X)$ as polynomials in $\sqrt{q}=X$,
$$f(1)=X^6 + 6X^5 + 10X^4 - 9X^2 - 2X + 1,$$
$$f'(1)=6X^5 + 25X^4 + 20X^3 - 18X^2 - 14X + 2,$$
and we write a Bézout relation between them:
$$p(X)f(1) +
q(X)f'(1) =
7,$$
with $p(X)=-12X^4 - 44X^3 - 27X^2 + 24X + 9$ and $q(X)=2X^5 + 11X^4 + 16X^3 - X^2 - 10X - 1$.
Then, considering the reduction modulo $7$ of $f'(1)$ and modulo $7$ and $49$ of $f(1)$, and applying the \emph{cyclicity criterion} (\ref{thm:l-cyclic}), we get that the class $\mathcal{I}_{\text{max}}^0(3,q)$ is cyclic.

For the second part of the theorem, note that there are only two polynomials without multiple roots, but reducible, that appear in the exhaustive research before the one corresponding to the class $\mathcal{I}_{\text{max}}^0(3,q)$ (see Table \ref{more_rational_points}). They correspond to the isogeny classes $\mathcal{E}_{max-1}(q)\times \mathcal{I}_{\text{max}}^0(2,q)$ and $\mathcal{E}_{max-2}(q)\times \mathcal{I}_{\text{max}}^0(2,q)$. In the first case, it is easy to check that $\mathcal{E}_{max-1}(q)$ and $\mathcal{I}_{\text{max}}^0(2,q)$ have coprime cardinality of their groups of rational points. Moreover, $\mathcal{I}_{\text{max}}^0(2,q)$ is cyclic and $\mathcal{E}_{max-1}(q)$ is eventually not cyclic only at the $3$--primary component, depending on $q$. In the second case, $\mathcal{E}_{max-2}(q)$ and $\mathcal{I}_{\text{max}}^0(2,q)$ are both cyclic. One can easily check applying the \emph{cyclicity criterion} (\ref{thm:l-cyclic}) that the corresponding isogeny class $\mathcal{E}_{max-2}(q) \times \mathcal{I}^0_{\text{max}}(2,q)$ of threefolds is cyclic too.
\end{proof}

\begin{remark}
Theorem \ref{ourfirsttheorem} fails when $q$ is not an even power of a prime. Indeed, for $q=599$ consider the $q$--Weil polynomial $$f(t)= t^6 + 142t^5 + 8516t^4 + 276053t^3 + 8516qt^2 + 142q^2t + q^3.$$ It corresponds to the maximal isogeny class of threefolds with endomorphism algebra a field. However, the class is ordinary but not cyclic, since $13\mid f'(1)$ and $13^2\mid f(1)$.
\end{remark}

\begin{remark}\label{Ng}
Using the notation introduced in the previous sections, the polynomial given in Theorem \ref{ourfirsttheorem} corresponds to having
\[\begin{cases}
r_1+r_2+r_3=-5\\
r_1r_2+r_1r_3+r_2r_3=6\\
r_1r_2r_3=-1
\end{cases}\]
and thus to the polynomial 
$$
\mathfrak{f}_3(t)=t^3-5t^2+6t-1.
$$
Then, using Corollary \ref{cor:main}, we easily obtain $N_3=-7^3$ and conclude that the isogeny class is cyclic outside the $7$--primary component. Indeed, the class is everywhere cyclic, as showed in Theorem \ref{ourfirsttheorem}.
\end{remark}
\begin{remark}\label{remark_LMFDB}
In \cite[\S 4.7]{LMFDB} the authors compute the Weil polynomials of the $\Fq$--maximal--simple isogeny class of abelian threefolds, for $q$ up to $25$. In particular, they point out that in their examples the simple--maximal class is always ordinary. However, they do not know how to explain this fact and are limited by the calculation power. We remark that the isogeny class with Weil polynomial as in Theorem \ref{ourfirsttheorem} is the simple--maximal one and it is proved to be ordinary.
\end{remark}

\noindent
\textbf{Acknowledgements.} Part of this work was completed while the first author was receiving funding from the French “Agence de l'innovation de défense” and from the French “Fondation Mathématiques Jacques Hadamard”.

\bibliography{BG-maximal_ab_var}
\bibliographystyle{siam}

\end{document}